\documentclass[11pt]{article}

\usepackage{amsmath,amssymb,xfrac}
\usepackage{algorithm,algorithmic}
\usepackage[margin=1in]{geometry}
\usepackage{url}
\usepackage{tikz}
\usetikzlibrary{arrows,calc,positioning,shapes.misc}

\usepackage{amsthm}

\newtheorem{lemma}{Lemma}

\theoremstyle{definition}

\usepackage{CustomCommands}

\title{A note on semi-infinite program bounding methods}
\author{%
Stuart M. Harwood\thanks{stuart.m.harwood@exxonmobil.com}%
\and Dimitri J. Papageorgiou%
\and Francisco Trespalacios\\
{\small Corporate Strategic Research}\\
{\small ExxonMobil Research and Engineering Company}\\
{\small Annandale, NJ 08801 USA}%
}
\date{\today}

\begin{document}
\maketitle

\begin{abstract}
Semi-infinite programs are a class of mathematical optimization problems with a finite number of decision variables and infinite constraints.
As shown by Blankenship and Falk~\cite{blankenshipEA76}, a sequence of lower bounds which converges to the optimal objective value may be obtained with specially constructed finite approximations of the constraint set.
In \cite{mitsos11}, it is claimed that a modification of this lower bounding method involving approximate solution of the lower-level program yields convergent lower bounds.
We show with a counterexample that this claim is false, and discuss what kind of approximate solution of the lower-level program is sufficient for correct behavior.
\end{abstract}

\section{Introduction}
This note discusses methods for the global solution of semi-infinite programs (SIP).
Specifically, the method from \cite{mitsos11} is considered, and it is shown with a counterexample that the lower bounds do not always converge.
Throughout we use notation as close as possible to that used in \cite{mitsos11}, embellishing it only as necessary with, for instance, iteration counters. 

Consider a SIP in the general form
\begin{alignat}{2} 
\tag{SIP}
\label{eq:sip}
f^* = 
\inf_{x}\; & f(x) \\
\st 				
\notag & x \in X, \\
\notag & g(x,y) \le 0,\; \forall y \in Y,
\end{alignat}
for subsets $X$, $Y$ of finite dimensional real vector spaces and 
$f : X \to \mbb{R}$,
$g : X \times Y \to \mbb{R}$.
We may view $Y$ as an index set, with potentially uncountably infinite cardinality.
Important to validating the feasibility of a point $x$ is the lower-level program:
\begin{equation}
\label{eq:llp}
\tag{LLP}
\sup_y \set{ g(x,y) : y \in Y}.
\end{equation}

Global solution of \eqref{eq:sip} often involves the construction of convergent upper and lower bounds.
The approach in \cite{mitsos11} to obtain a lower bound is a modification of the constraint-generation/discretization method of \cite{blankenshipEA76}.
The claim is that the lower-level program may be solved approximately;
the exact nature of the approximation is important to the convergence of the lower bounds and this is the subject of the present note.

%

\section{Sketch of the lower bounding procedure and claim}

The setting of the method is the following.
The method is iterative and at iteration $k$, for a given finite subset $Y^{LBD,k} \subset Y$, a lower bound of $f^*$ is obtained from the finite program
\begin{alignat}{2} 
\label{eq:sip_lower}
f^{LBD,k} = 
\inf_{x}\; & f(x) \\
\st 				
\notag & x \in X, \\
\notag & g(x,y) \le 0, \;\forall y \in Y^{LBD,k}.
\end{alignat}
This is indeed a lower bound since fewer constraints are enforced, and thus \eqref{eq:sip_lower} is a relaxation of \eqref{eq:sip}.
Assume that the lower bounding problem \eqref{eq:sip_lower} is feasible
(otherwise we can conclude that \eqref{eq:sip} is infeasible).
Let $\bar{x}^k$ be a (global) minimizer of the lower bounding problem \eqref{eq:sip_lower}.
In \cite{mitsos11}, Lemma~2.2 states that we either verify 
$\sup_y \set{g(\bar{x}^k,y) : y \in Y} \le 0$,
\textbf{or else} find $\bar{y}^k \in Y$ such that $g(\bar{x}^k,\bar{y}^k) > 0$.
If $\sup_y \set{g(\bar{x}^k,y) : y \in Y} \le 0$,
then $\bar{x}^k$ is feasible in \eqref{eq:sip} and thus optimal (since it also solves a relaxation).
Otherwise, set $Y^{LBD,k+1}  = Y^{LBD,k} \cup \set{\bar{y}^k}$ and we iterate.

The precise statement of the claim is repeated here (again, with only minor embellishments to the notation to help keep track of iterations).

\begin{lemma}[Lemma~2.2 in \cite{mitsos11}]
\label{lem:claim}
Take any $Y^{LBD,0} \subset Y$.
Assume that $X$ and $Y$ are compact and that $g$ is continuous on $X \times Y$.
Suppose that at each iteration of the lower bounding procedure the lower-level program is solved approximately for the solution of the lower bounding problem $\bar{x}^k$ either establishing 
$\max_{y \in Y} g(\bar{x}^k,y) \le 0$, or furnishing a point $\bar{y}^k$ such that $g(\bar{x}^k,\bar{y}^k) > 0$.
Then, the lower bounding procedure converges to the optimal objective value, i.e. $f^{LBD,k} \to f^*$.
\end{lemma}

\section{Correction}

\subsection{Counterexample}

We now present a counterexample to the claim in Lemma~\ref{lem:claim}.
Consider 
\begin{alignat}{2} 
\tag{CEx}
\label{eq:counter_ex}
\inf_{x}\; & -x \\
\st 				
\notag & x \in [-1,1], \\
\notag & 2x - y \le 0,\; \forall y \in [-1,1],
\end{alignat}
thus we define 
$X = Y = [-1,1]$, 
$f : x \mapsto -x$,
$g : (x,y) \mapsto 2x - y$.
The behavior to note is this:
We are trying to maximize $x$;
The feasible set is
\[
\set{x \in [-1,1] : x \le (\sfrac{1}{2})y, \forall y \in[-1,1] } = [-1,-\sfrac{1}{2}];
\]
The infimum, consequently, is $\sfrac{1}{2}$.
See Figure~\ref{fig:cex1}.

\begin{figure}
\begin{center}
\begin{tikzpicture}[xscale=1.5,yscale=1.5]
\draw[fill=gray] (-0.5,-1) -- (0.5,1) -- (1,1) -- (1,-1) -- cycle;
\draw[latex-latex] (0,-1.3) -- (0,1.3) node[above]{$y$}; 
\draw[latex-latex] (-1.3,0) -- (1.3,0) node[right]{$x$}; 
\draw (-1,-1) rectangle (1,1); 
\draw[dashed,domain=0:1] plot (\x, \x);
\end{tikzpicture}
\end{center}
\caption{Visualization of counterexample~\eqref{eq:counter_ex}.
The box represents $[-1,1] \times [-1,1]$.
The shaded grey area is the subset of $(x,y)$ such that $2x - y > 0$.
The dashed line represents the approximate minimizers used in the counterexample.}
\label{fig:cex1}
\end{figure}
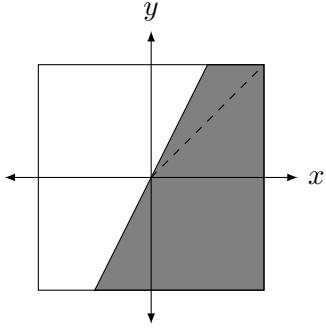

Beginning with $Y^{LBD,1} = \emptyset$,  the minimizer of the lower bounding problem is $\bar{x}^1 = 1$.
Now, assume that solving the resulting \eqref{eq:llp} approximately, we get $\bar{y}^1 = 1$ which we note satisfies
\[
2\bar{x}^1 - \bar{y}^1 = 1 > 0
\]
as required by Lemma~\ref{lem:claim}.

The next iteration, with $Y^{LBD,2} = \set{1}$, adds the constraint 
$2x - 1 \le 0$
to the lower bounding problem;
the feasible set is $[-1,\sfrac{1}{2}]$ so the minimizer is $\bar{x}^2 = \sfrac{1}{2}$.
Again, assume that solving the lower-level program approximately yields $\bar{y}^2 = \sfrac{1}{2}$;
again we get
\[
2\bar{x}^2 - \bar{y}^2 = \sfrac{1}{2} > 0
\]
as required by Lemma~\ref{lem:claim}.

The third iteration, with $Y^{LBD,3} = \set{1, \sfrac{1}{2}}$, adds the constraint 
$2x - \sfrac{1}{2} \le 0$
to the lower bounding problem;
the feasible set is $[-1,\sfrac{1}{4}]$ so the minimizer is $\bar{x}^3 = \sfrac{1}{4}$.
Again, assume that solving the lower-level program approximately yields $\bar{y}^3 = \sfrac{1}{4}$;
again we get
\[
2\bar{x}^3 - \bar{y}^3 = \sfrac{1}{4} > 0
\]
as required by Lemma~\ref{lem:claim}.

Proceeding in this way, we construct $\bar{x}^k$ and $\bar{y}^k$ so that 
$g(\bar{x}^k,\bar{y}^k) > 0$ and 
the lower bounds satisfy 
$f^{LBD,k} = -\bar{x}^k = -\frac{1}{2^{k-1}}$, for all $k$.
Consequently, they converge to $0$, which we note is strictly less than the infimum of $\sfrac{1}{2}$.

\subsection{Modified claim}

We now present a modification of the claim in order to demonstrate what kind of approximate solution of the lower-level program suffices to establish convergence of the lower bounds.
To state the result, let the optimal objective value of \eqref{eq:llp} as a function of $x$ be
\[
g^*(x) = \sup_y\set{g(x,y) : y \in Y}.
\]
The proof of the following result has a similar structure to the original proof of \cite[Lemma~2.2]{mitsos11}.

\begin{lemma}
\label{lem:claim_mod}
Choose any finite $Y^{LBD,0} \subset Y$, and $\alpha \in (0,1)$.
Assume that $X$ and $Y$ are compact and that $f$ and $g$ are continuous.
Suppose that at each iteration $k$ of the lower bounding procedure \eqref{eq:llp} is solved approximately for the solution $\bar{x}^k$ of the lower bounding problem~\eqref{eq:sip_lower}, either establishing that
$g^*(\bar{x}^k) \le 0$ 
or
furnishing a point $\bar{y}^k$ such that 
\[
g(\bar{x}^k,\bar{y}^k) \ge \alpha g^*(\bar{x}^k) > 0.
\]
Then, the lower bounding procedure converges to the optimal objective value, i.e. $f^{LBD,k} \to f^*$. 
\end{lemma}
\begin{proof}
First, if the lower bounding problem~\eqref{eq:sip_lower} is ever infeasible for some iteration $k$, then \eqref{eq:sip} is infeasible and we can set $f^{LBD,k} = +\infty = f^*$.
Otherwise, since $X$ is compact, $Y^{LBD,k}$ is finite, and $f$ and $g$ are continuous, for every iteration the lower bounding problem has a solution by Weierstrass' (extreme value) theorem.
If at some iteration $k$ the lower bounding problem furnishes a point $\bar{x}^k$ for which $g^*(\bar{x}^k) \le 0$, then $\bar{x}^k$ is feasible for \eqref{eq:sip}, and thus optimal.
The corresponding lower bound $f^{LBD,k}$, and all subsequent lower bounds, equal $f^*$.

Otherwise, we have an infinite sequence of solutions to the lower bounding problems.
Since $X$ is compact we can move to a subsequence $\seq[k \in \mbb{N}]{\bar{x}^{k}} \subset X$ which converges to $x^* \in X$.
By construction of the lower bounding problem we have
\[
g(\bar{x}^{\ell},\bar{y}^k) \le 0, \quad \forall \ell,k : \ell > k.
\]
By continuity and compactness of $X \times Y$ we have uniform continuity of $g$, and so for any $\epsilon > 0$, there exists a $\delta > 0$ such that
\begin{equation}
\label{eq:gee}
g(x,\bar{y}^k) < \epsilon,
	\quad\forall x : \norm{x - \bar{x}^{\ell}} < \delta, 
		\quad\forall \ell,k : \ell > k.
\end{equation}
Since the (sub)sequence $\seq[k \in \mbb{N}]{\bar{x}^k}$ converges, there is an index $K$ sufficiently large that 
\begin{equation}
\label{eq:tails}
\norm{\bar{x}^{\ell} - \bar{x}^k} < \delta, \quad \forall \ell,k : \ell > k \ge K.
\end{equation}
Using \eqref{eq:tails}, we can substitute $x = \bar{x}^k$ in \eqref{eq:gee} to get that for any $\epsilon > 0$, there exists $K$ such that
\[
g(\bar{x}^k,\bar{y}^k) < \epsilon, \quad \forall k \ge K.
\]
By assumption $g(\bar{x}^k,\bar{y}^k) > 0$ for all $k$, and so combined with the above we have that
$g(\bar{x}^k,\bar{y}^k) \to 0$. 

Combining $g(\bar{x}^k,\bar{y}^k) \to 0$ with
$
g(\bar{x}^k,\bar{y}^k) \ge \alpha g^*(\bar{x}^k) > 0,
$
for all $k$,
we see 
$
g^*(\bar{x}^k) \to 0.
$
Meanwhile $g^* : X \to \mbb{R}$ is a continuous function, by classic parametric optimization results like \cite[Theorem~1.4.16]{aubin_frankowska} (using continuity of $g$ and compactness of $Y$).
Thus 
\[
g^*(x^*)  = \lim_{k \to \infty} g^*(\bar{x}^k) = 0.
\]
Thus $x^*$ is feasible in \eqref{eq:sip} and so $f^* \le f(x^*)$.
But since the lower bounding problem is a relaxation, $f^{LBD,k} = f(\bar{x}^k) \le f^*$ for all $k$, and so by continuity of $f$, $f(x^*) \le f^*$.
Combining these inequalities we see $f^{LBD,k} \to f(x^*) = f^*$.
Since the entire sequence of lower bounds is an increasing sequence, we see that the entire sequence converges to $f^*$ (without moving to a subsequence).
\end{proof}
\section{Remarks}

The main contribution of \cite{mitsos11} is a novel \emph{upper} bounding procedure, which still stands, and combined with the modified lower bounding procedure from Lemma~\ref{lem:claim_mod} or the original procedure from \cite{blankenshipEA76}, the overall global solution method for \eqref{eq:sip} is still effective.

The counterexample that has been presented may seem contrived.
However, as the lower bounding method for SIP from \cite{mitsos11} is adapted to give a lower bounding method for \emph{generalized} semi-infinite programs (GSIP) in \cite{mitsosEA15}, a modification of the counterexample reveals that similar behavior may occur (and in a more natural way) when constructing the lower bounds for a GSIP.
Consequently, the lower bounds fail to converge to the infimum.
See \cite{Harwood19_GSIP}.


\end{document}